\documentclass[11pt,english]{article}
\usepackage[T1]{fontenc}
\usepackage[latin9]{inputenc}
\usepackage{amsmath}
\usepackage{amsthm}
\usepackage{amssymb}

\makeatletter
\theoremstyle{plain}
\newtheorem{thm}{\protect\theoremname}[section]
\theoremstyle{plain}
\newtheorem{cor}[thm]{\protect\corollaryname}
\theoremstyle{plain}
\newtheorem{lem}[thm]{\protect\lemmaname}
\theoremstyle{definition}
\newtheorem{problem}[thm]{\protect\problemname}

\DeclareMathOperator{\End}{End}
\DeclareMathOperator{\Aut}{Aut}
\DeclareMathOperator{\supp}{supp}
\date{}

\makeatother

\usepackage{babel}
\providecommand{\corollaryname}{Corollary}
\providecommand{\lemmaname}{Lemma}
\providecommand{\problemname}{Problem}
\providecommand{\theoremname}{Theorem}

\begin{document}
\title{On the analytic self-maps of Cantor sets in the line}
\author{Michael Hochman\thanks{Supported by ISF grants 1702/17 and 3056/21}}
\maketitle
\begin{abstract}
We show that a topological Cantor set in $\mathbb{R}$ has at most
countably many real-analytic, onto self-maps.
\end{abstract}

\section{Introduction}

A Cantor set is a non-empty, compact, totally disconnected set with
no isolated points. If $X,Y\subseteq\mathbb{R}$ are Cantor sets and
$\alpha\geq0$, we define \footnote{We point out that analyticity implies that all extensions of a map
$f\in\End^{\omega}(X)$ to an interval containing $X$ are compatible
(in contrast, when $1\leq\alpha\leq\infty$, the extension is never
unique).}
\[
\End^{\alpha}(X,Y)\left\{ f:X\rightarrow Y\;\left|\negthickspace\begin{array}{c}
f(X)=Y\text{ and }f\text{ extends to a }C^{\alpha}\\
\text{ map of an open interval containing}X
\end{array}\right.\right\} 
\]
We abbreviate $\End^{\alpha}(X)=\End^{\alpha}(X,X)$ for the semigroup
of $C^{\alpha}$ onto self-maps.

In general, the size of $\End^{\omega}(X)$ is affected both by $\alpha$
and $X$. Every Cantor set $X\subseteq\mathbb{R}$ has uncountably
many $C^{0}$-self maps, but in the regime $0<\alpha\leq\infty$,
the set $\End^{\alpha}(X)$ may be empty, finite, countably infinite
or uncountable, depending on the structure of $X$. 

The main result of this note is the observation that the real-analytic
case $\alpha=\omega$ is much more constrained:
\begin{thm}
\label{thm:main}If $X,Y\subseteq\mathbb{R}$ is a Cantor set, then
$\End^{\omega}(X,Y)$ is at most countable.
\end{thm}

This result seems to be unknown in the fractal geometry community,
but after this paper was completed we learned that the main step of
the argument, which deals with increasing maps of an interval, can
be derived from Proposition 3.5 of \cite{Navas2006}. Our proof, which
is based on different considerations, is more elementary and self-contained
(e.g. it does not rely on Szekeres's theorem).

It is crucial to the statement that we consider analytic maps defined
on intervals containing $X$, rather than, say, those defined on a
(possibly disconnected) open neighborhood, i.e. piecewise analytic
maps. For example, on the middle-$1/3$ Cantor set $C\subseteq[0,1]$,
define the map that takes $C\cap[0,1/3]$ onto $C$ and maps $C\cap[2/3,1]$
identically to a fixed $x_{0}\in C$. Such maps are analytic on a
neighborhoo of $C$ but there are uncountably many choices for $x_{0}$.
Nevertheless, under the assumption that the maps are open, or open
outside a finite set, one can deduce the following:
\begin{cor}
If $X,Y\subseteq\mathbb{R}$ are Cantor sets, then there are at most
countably many maps $f:X\rightarrow Y$ that are open on the complement
of some finite subset of $X$, and that extend analytically to a neighborhood
of $X$.
\end{cor}

\begin{proof}
If $F\subseteq X$ is finite and $f:U\rightarrow\mathbb{R}$ is an
open analytic map on a neighborhood $U$ of $X\setminus F$, then
in every connected component $U'\subseteq U$ we can find Cantor sets
$X',Y$ that are open and closed in $X,Y$, respectively, and such
that $f|_{X'}\in\End^{\omega}(X',Y')$, and $f|_{X'}$ determines
$f|_{U'}$ . The corollary follows since there are countably many
possibilities for each choice above.
\end{proof}

We emphasize again that we make no assumptions on $X$ beyond the
topological one. If one assumes that $X$ has more structure, such
results are well known and one can say a lot more. For example, let
$X$ be the central-$\alpha$ Cantor, constructed from $[0,1]$ by
removing the open middle segment of relative length $\alpha$ and
iterating the procedure. Then every affine map taking $X$ into itself
contracts by an integer power of $1-\alpha$, see \cite{ElekesKeletiMathe2010}.
Together with a linearization argument, this implies that if $f:X\rightarrow X$
is a local diffeomorphism, then the derivative $f'$ takes values
that are rational powers of $1-\alpha$, and in particular, if $f\in C^{\omega}$,
it is affine. A similar result holds if $X\subseteq[0,1]$ is a Cantor
set invariant and transitive under $\times a\bmod1$ for an integer
$a\geq2$, see \cite{Hochman2018-smooth-symmetries-of-xa-invariant-sets}.
More generally, Funar and Neretin showed that if $X$ is porous (``sparse''
in their terminology) then $\End^{\alpha}(X)$ is countable. But they
also showed that there are Cantor sets such that $\End^{\infty}(X)$
is uncountable \cite[Theorem 2,3 and Section 5.1]{FunarNeretin2018}
(note that Theorem 1 of \cite{FunarNeretin2018} concerns a completely
different space than Theorem \ref{thm:main} and does not imply it). 

\section{Proof }

In this section we prove Theorem (\ref{thm:main}). The proof consists
of a few reductions, followed by a combinatorial analysis of the action
of the map on the gaps in the Cantor set $X$.

\subsubsection*{Reduction to non-singular orientation-preserving bijections}

Fix Cantor sets $X,Y\subseteq\mathbb{R}$. We first use analyticity
and a topological argument to go from onto maps to invertible ones
(for open maps, this step is trivial).
\begin{lem}
\label{lem:Baire-argument}If $f\in\End^{\omega}(X,Y)$, then there
exist open intervals $I,J\neq0$ such that
\begin{itemize}
\item $I\cap X\neq\emptyset$ is open and closed in $X$,
\item $f$ extends to an analytic map $I\rightarrow J$.
\item $f'\neq0$ on $I$, and in particular, $f:I\rightarrow J$ is bijective.
\item $f(X\cap I)=Y\cap J$, and in particular, $f(X\cap I)$ is open and
closed in $Y$.
\end{itemize}
\end{lem}

\begin{proof}
By definition we can assume that $f$ is analytic in an open neighborhood
$U$ of $X$. We can find a finite or countable cover of $U$ by closed
intervals $\{I_{i}\}$ on which $f$ is strictly monotone. Then $f(I_{i}\cap X)$
is a cover of $X$ by closed sets, so, by Baire's theorem, one of
them, say $f(I_{i}\cap X)$, must have non-empty interior in $X$.
Let $J\subseteq f(I_{i})$ be an open interval such that
\begin{enumerate}
\item $J\cap X\neq\emptyset$
\item The endpoints of $J$ are in the complement sof $X$.
\item $J$ is disjoint from the images of point where $f'$ vanishes (there
are only finitely many such point in $I$, for otherwise $f$ would
be constant).
\end{enumerate}
Now $J$ and $I=f^{-1}J$ are the desired intervals.
\end{proof}
Let $I,J$ be intervals as in the lemma. Observe that knowing $f|_{I\cap X}$
determines $f$ uniquely on $I$, because $f$ is analytic and $I\cap X$
is infinite. Thus, it suffices for us to show that for every pair
of closed and open sets $Y,Z\subseteq X$ and every pair of open intervals
$I,J$ containing them, respectively, there are only countably many
real-analytic maps $I\rightarrow J$ such that $f(I\cap X)=J\cap X$
and $f'\neq0$ on $I\cap X$.

Moreover, given such $Y,Z$ and $f$, by considering $h=g^{-1}\circ f$
as $g$ ranges over all others such maps, we find that it suffices
to prove the statement in the last paragraph in the case $Y=Z$ and
$I=J$.

We can reduce further to the case that $f$ is orientation preserving
($f'>0$). Indeed, if $h:I\rightarrow I$ were one such orientation-reversing
map, then we could again consider the family $g\circ h$ as $g$ ranges
over all maps with the same property; this family is in bijection
with the orientation-reversing maps but consists of orientation-preserving
maps, so if the latter is countable, so is the former.

\subsubsection*{Reductions to concave maps with attracting fixed points}

Fix $Y,I,h$ be as above, with $h$ orientation-preserving ($h'>0$
on $I$). Observe that $h$ has fixed points in $I$; indeed, the
maximal and minimal points in $Y\cap I$ are fixed by $h$, since
$h$ is monotone and maps $Y\cap I$ onto itself. Furthermore, it
has finitely many fixed points, since the only analytic map fixing
infinitely many points in an interval is the identity map, and we
may exclude this specific case.

Replace $I=[a.b]$ by a minimal non-trivial closed sub-interval of
$I$ that intersects $Y$ in a Cantor set and whose endpoints are
fixed points of $h$. We can achieve this by dissection: if $p\in I$
is an interior fixed point then one of the sub-intervals $[a,p].[p,b]$
is a Cantor set. Repeat until a minimal interval is reached.

Having re-defined $I=[a,b]$ as above, note that either $h(x)>x$
in the interior of $I$, or $h(x)<x$ in the interior of $I$.

We may also assume $b$ to be an attracting fixed-point of $h$, i.e.
that we are in the case $h(x)>x$ in the interior of $I$. Indeed,
$b$ is a fixed point, and if we were in the case $h(x)<x$ we can
replace $h$ by $h^{-1}$.

Finally, by analyticity again, there is a left-neighborhood $(b-\delta,b]$
of $b$ in which $h$ is concave, or one where $h$ is convex. We
can assume We first discuss the concave case, and return to the convex
case at the end of proof.

We have arrived in the following situation. We have an interval $U=[b-\delta,b]$
with $b\in X$ such that $U\cap X$ is oen and closed in $X$, and
hence a Cantor set. Denote by $\mathcal{F}_{U}\subseteq\End^{\omega}(X)$
the set of real-analytic monotone orientation-preserving concave maps
$h$ of $I$ that fix $b$, map $U$ into itself, map $X\cap U$ onto
$X\cap f(U)$, satisfy $h(x)>x$ in the interior of $U$, and $f'(b)>0$.

Our goal is to show $\mathcal{F}_{U}$ is countable. 

\subsubsection*{The action of $h$ on intervals and distortion bounds}

We need two elementary properties of the action of $h$ on intervals.
Let $I\subseteq U$ be an interval and let $|I|$ denote its length.
We have a trivial bound,
\begin{equation}
\inf_{x\in U}|h'(x)|\leq\frac{|h(I)|}{|I|}\leq\sup_{x\in U}|h'(x)|\label{eq:lengths-ratios-converge}
\end{equation}
Since $h$ is concave on $U$, the left hand side is equal to $C=h'(b)>0$.
We shall assume that we have a lower bound on $C$. More precisely,
we shall show that for each $M\in\mathbb{N}$, there are only countably
many relevant $h$ for which $C\geq1/M$.

We partially order intervals according to whether one is entirely
to one side of the other in the line: If $I_{1},I_{2}\subseteq U$
are intervals, write $I_{1}\leq I_{2}$ provided that $x\leq y$ for
every $x\in I_{1}$, $y\in I_{2}$. 

Concavity then implies that 
\begin{equation}
I_{1}\leq I_{2}\quad\implies\quad\frac{|h(I_{2})|}{|h(I_{1})|}\leq\frac{|I_{2}|}{|I_{1}|}\label{eq:interval-ratio}
\end{equation}

\subsubsection*{Leading gaps, and the action of $h$ on them}

By a \textbf{gap} we shall mean a maximal open interval in $U\setminus X$.

Since $X$ is a Cantor set and intersects the interior of $U$ non-trivially,
the set of gaps is countably infinite, and the order type of the set
of gaps (in the partial order defined above) is dense.

Since $h:U\rightarrow U$ is strictly increasing and maps $U\cap X$
onto itself, it induces an order-preserving injection of the set of
gaps. Note that the converse is also true: the induced map on gaps
determines the action of $h$ on the endpoints of each gap, and since
$h$ is real-analytic, this determines $h$. In fact, $h$ is uniquely
determined by its action on any infinite set of gaps.

It is not hard to see that there are uncountably many order-preserving
injections on the set of gaps. Our goal is to show that only countably
many can come from $C^{\omega}$-self maps of $X$. The proof is based
on identifying a distinguished sequence of gaps on which the action
of $h$ is highly constrained. 

To this end, we say that an interval $J\subseteq(b-\delta,b)$ is
a \textbf{leading gap }if
\begin{enumerate}
\item [(I)]$J$ is a gap,
\item [(II)]Every gap that lies to the right of $J$ is strictly shorter
than $J$.
\end{enumerate}
It follows from (\ref{eq:interval-ratio}) that, if $J$ is a leading
leading gap, then so is $h(J)$. 

Fix a leading gap $J_{1}\subseteq(b-\delta,b)$ and write $J_{n}=h^{n}(J_{1})$.
Then all $J_{n}$ are leading gaps. Furthermore, the sequence $h^{n}(b-\delta)$
of left endpoints of $J_{n}$ is strictly increasing because $h(x)>x$,
and different gaps are disjoint; so it follows that $J_{n}\leq J_{n+1}$.

Let $J'_{n}$ denote the closed interval sandwiched between $J_{n}$
and $J_{n+1}$ (this is not a gap!). Evidently, 
\[
J'_{n+1}=h(J'_{n})
\]
If $J'_{n}=\{x\}$ then $x$ is isolated in $X$, contrary to assumption
that $X$ is a Cantor set, so $|J'_{n}|>0$ for all $n$. 

Writing $C>0$ for the constant on the left hand side of (\ref{eq:lengths-ratios-converge}),
and using (\ref{eq:interval-ratio}), we conclude that
\begin{equation}
\frac{|J'_{n}|}{|J_{n+1}|}\leq C\cdot\frac{|J'_{n}|}{|J_{n}|}=C\cdot\frac{|h^{n}(J'_{1})|}{|h^{n}(J_{1})|}\leq C\cdot\frac{|J'_{1}|}{|J_{1}|}\label{eq:;leng-ratio-of-leader-to-gap}
\end{equation}
Note that the right hand side is independent of $n$.

Although leading gaps are mapped to leading gaps, the sequence $\{J_{n}\}$
need not contain all leading gaps that lie on the right of $J_{1}$,
as some leading gaps may be contained in the $J'_{n}$s. However,
every leading gap in $J'_{n}$ is mapped by $h$ to a leading gap
in $J'_{n+1}$, so if we denote by $k(n)$ the number of leading gaps
in $J'_{n}$, then $\{k(n)\}$ is a non-decreasing sequence. 

We claim that $k(n)$ is also a bounded sequence, because every leading
gap in $J'_{n}$ is at least as long as $J_{n+1}$ (by the leading
property), and so $k(n)\leq|J'_{n}|/|J_{n+1}|=O(1)$ by (\ref{eq:;leng-ratio-of-leader-to-gap}). 

To conclude the argument, observe that $J_{1}=(b-\delta,b)$ and $(k(n))_{n=1}^{\infty}$
determine the entire sequence $(J_{n})$, and this determines an infinite
orbit of $h$, for example $h^{n}(b-\delta)$. Thus if $J_{1},(k(n))$
are known then $h$ is known on an infinite sequence and is thus uniquely
determined.. There are only countably many leading gaps $J_{1}$ and
countably many non-decreasing bounded integer sequences, so we have
shown that $\mathcal{F}_{U}$ is countable.

\subsubsection*{The convex case}

In the concave case the proof proceeds along similar lines. We have
an interval $U=[b-\delta,b]$ with $b\in X$ such that $U\cap X$
is a Cantor set. Denote by $\mathcal{G}_{U}\subseteq\End^{\omega}(X)$
the set of real-analytic monotone orientation-preserving convex maps
$h$ of $I$ that fix $b$. map $U$ to itself, map $X\cap U$ onto
itself, satisfy $h(x)>x$ in the interior of $U$, and such that $f'(b)\neq0$. 

Fix an interval $U=[b-\delta,b]$, let $\mathcal{G}_{U}\subseteq\End^{\omega}(X)$
denote the set of real-analytic self-maps $h$ that map $U$ onto
itself, map $X\cap U$ onto itself, whose endpoints lie in $X$ and
are fixed by $h$, such that $h$ is increasing and convex, satisfies
$h(x)>x$ in the interior of $U$, and $h'(b)>0$ in the interior.
We must show that $\mathcal{G}_{U}$ is countable. 

Since $h^{-1}$ is concave, we know by the previous argument that
if $J\subseteq U$ is a leading gap and if $h^{-1}(J)\subseteq U$,
then $h^{-1}(J)$ is also a leading gap. But note that it could happen
that $h(J)$ is not.

Nevertheless, for each leading gap $J\subseteq U$ we can produce
a maximally long sequences $(h^{-n}(J))_{n=0}^{N}$ of leading gaps,
terminating when $h^{-N}(J)\not\subseteq h(U)$. The same analysis
as before shows that the ratios of lengths of $h^{-(n+1)}(J),h^{-1}(J)$
and of the interval between them are bounded. It follows that for
every pair $h^{-n}(J)$, $h^{-(n+1)}(J)$ of such a sequence, there
is at most a bounded number $M$ of leading gaps between them. It
also follows that the terminating gap $h^{-N}(J)$ belongs to a finite
set of at most $M$ leading gaps (perhaps the constant is a-priori
different, but we choose an $M$ that is good for both bounds). Finally,
by taking $J$ close enough to $b$, we can ensure that $N$ is arbitrarily
large.

It follows that we can choose a sequence of initial leading gaps $J$
for which the resulting sequences extend each other, so we can form
at least one infinite such sequence, which, by reversing the order,
has the form $(h^{n}(J))_{n=0}^{\infty}$. There are at most $M$
possibilities for $J$, independent of $h$. 

The argument is completed as before, the only different being that
$k(n)$ is a non-increasing sequence of natural numbers, rather than
a bounded non-decreasing one.

\subsubsection*{Remarks}

Our proof gives a little more than was stated: In the end of thee
argument It is clear that each map $h$ is characterized by a leading
gap and a tail of the sequence $k(n)$ that is constant. Then it is
easy to see that the set of $m$ that can be the tail value is a cyclic
sub-semigroup of $\mathbb{N}$, and consequently, the set of such
$h$ is cyclic.

Also note, that we used analyticity in two rather mild ~ways;

First, it guaranteed that each map is determined by its behavior on
sufficiently many (in fact, countably many) points. Some such assumption
is unavoidable. For example, one can construct a Cantos set $X$ and
disjoint intervals $I_{1}\leq I_{2}\leq\ldots$ such that $X=\bigcup(X\cap I_{n})\cup\{x_{0}\}$.
Assuming that $X\cap I_{n}$ has $C^{\infty}$ self-maps $f_{n}\neq g_{n}$
close enough to the identity map, we can form a self-map by choosing
$h_{n}\in\End^{\infty}(X\cap I_{n})$ close to the identity, and setting
$h|_{I_{n}}\in\{f_{n},g_{n}\}$ and $h(x_{0})=x_{0}$. Then $\End^{\infty}(X)$
contains a copy of the product $\prod\{f_{n},g_{n}\}$, so it is uncountable.

Second, we used analyticity to ensure that the maps are piecewise
monotone and piecewise concave/convex. Some condition like this is
also necessary, because Funar and Neretin have constructed an example
of a Cantor set $X\subseteq\mathbb{R}$ with uncountably many $C^{\infty}$
maps fixing a given point in $X$ \cite{FunarNeretin2018}. 

\section{Problems}

Our main theorem suggests many generalizations but we have not succeeded
in proving or disproving any of them.
\begin{problem}
What can be said about the set of into (rather than onto) self- maps
of a Cantor set $X\subseteq\mathbb{R}$?
\end{problem}

The following related question was raised in \cite{Hochman2017},
but remains open:
\begin{problem}
If a Cantor set $X\subseteq\mathbb{R}$ has dimension $<1$ can there
exist uncountably many affine self-maps of $X$ into itself?
\end{problem}

Another direction concerns the algebraic structure of $\End^{\omega}(X)$:
\begin{problem}
Let $\Aut^{\omega}(X)$ be the invertible elements of $\End^{\omega}(X)$.
What groups arise as {[}subgroups{]} if $\Aut^{\omega}(X)$? When
is $\Aut^{\omega}(X)$ it finitely generated?
\end{problem}

For certain structured sets this question was addressed by Funar and
Neretin \cite{FunarNeretin2018}.

It is also natural to consider higher dimensions. Certainly our proof
breaks down, as it relies very strongly on the order structure of
$\mathbb{R}$. One also must adjust the statement, since if $X$ is
contained in a lower-dimensional $C^{\alpha}$-manifold $M$, one
may be able to extend a self-map of $X$ in many ways in the transverse
direction to the manifold. 
\begin{problem}
Is the analogue of Theorem \ref{thm:main} true for $X\subseteq\mathbb{R}^{d}$
and real-analytic maps, provided $M\cap X$ has empty interior in
$X$ for every $C^{\omega}$-submanifold $M\subseteq\mathbb{R}^{d}$?
\end{problem}

In the complex case, the qualification is unnecessary:
\begin{problem}
Is the analogue of Theorem \ref{thm:main} true for Cantor sets $X\subseteq\mathbb{C}$
and analytic maps? 
\end{problem}

And finally, going back to the line, 
\begin{problem}
\label{prob:measures}If $\mu$ is a compactly supported, non-atomic
Borel probability measure on $\mathbb{R}$, is 
\[
\End^{\omega}(\mu)=\{f\in\End^{\omega}(\supp\mu)\mid\mu=\mu\circ f^{-1}\}
\]
countable? 

Of course, if the support is a Cantor set, this is answered by the
results of the present paper (because any map preserving $\mu$ must
preserve its support), but if the support an interval the answer is
far from clear. The problem may be viewed as a relative of Furstenberg's
celebrated $\times2,\times3$ problem, but the two are not directly
related since here we consider more general measures and maps but
ask for a weaker conclusion. We remark, that for $\times m$ invariant
measures of positive dimension, we resolved Problem \ref{prob:measures}
in \cite{Hochman2018-smooth-symmetries-of-xa-invariant-sets}. It
is possible that also the problem above becomes easier under the assumption
that $\dim\mu>0$ even this remains open.

\bibliographystyle{plain}
\bibliography{bib}

\bigskip
\bigskip
\footnotesize
\noindent{}\texttt{Department of Mathematics, The Hebrew University of Jerusalem, Jerusalem, Israel\\ email: michael.hochman@mail.huji.ac.il}
\end{problem}

\end{document}